\documentclass[11pt,reqno]{amsart}
\numberwithin{equation}{section}
\usepackage{amssymb,amsmath}
\usepackage{amsmath,amsfonts,amsthm,mathrsfs}
\usepackage{color,graphics,epsfig}

 \def\bsr{\operatorname{bsr}}
  \def\tsr{\operatorname{tsr}}

\newcommand{\mC}{\mathbb{C}}

\newcommand{\mN}{\mathbb{N}}

\newcommand{\mQ}{\mathbb{Q}}
\newcommand{\mR}{\mathbb{R}}

\newtheorem{theorem}{Theorem}[section]
\newtheorem{lemma}[theorem]{Lemma}
\newtheorem{corollary}[theorem]{Corollary}
\newtheorem{proposition}[theorem]{Proposition}

\theoremstyle{definition}

\newtheorem{remark}[theorem]{Remark}

\theoremstyle{definition}
\newtheorem{definition}[theorem]{Definition}

\theoremstyle{definition}

\begin{document}



\title[Stable ranks of the Bohl algebra]{The Bass and topological stable ranks of the Bohl algebra are infinite}
\author{Raymond Mortini}
\address{
Universit\'{e} de Lorraine\\ D\'{e}partement de Math\'{e}matiques et  
Institut \'Elie Cartan de Lorraine,  UMR 7502\\
 Ile du Saulcy\\
 F-57045 Metz, France} 
 \email{Raymond.Mortini@univ-lorraine.fr}
\author{Rudolf Rupp}
\address{Fakult\"at f\"ur Angewandte Mathematik, Physik  und Allgemeinwissenschaften\\
\small TH-N\"urnberg\\
\small Kesslerplatz 12\\
\small D-90489 N\"urnberg, Germany}
\email  {Rudolf.Rupp@th-nuernberg.de}
\author{Amol Sasane}
\address{Mathematics Department \\ London School of Economics\\ Houghton Street \\ London WC2A 2AE \\ United Kingdom.}
\email{sasane@lse.ac.uk}

\begin{abstract}
The Bohl algebra $\textrm{B}$ is the ring of linear combinations of functions 
 $t^k e^{\lambda t}$, where $k$ is any nonnegative integer, and $\lambda$ is 
any complex number,  with pointwise operations. We show that the Bass stable rank 
 and the topological stable rank of $\textrm{B}$ (where we use the topology of uniform convergence) 
are infinite. 
\end{abstract}

\maketitle

\section{Introduction}

\noindent The aim of this article is to investigate the algebraic-analytic properties of the Bass stable rank and 
topological stable rank, which are concepts arising from algebraic and topological $K$-theory, for a 
particular algebra of functions, called the Bohl algebra. We give the pertinent definitions below. 

\begin{definition}[Bohl algebra $\textrm{B}$]
Let $\mN:=\{0,1,2,3,\cdots\}$ denote the set of nonnegative integers, and 
$\mN^*:=\{1,2,3,\cdots\}$. 
 The {\em Bohl algebra} $\textrm{B}$ is the complex algebra of functions on the real line $\mR$,
 $$
 \textrm{B}:=\bigg\{ \sum_{n=1}^N c_n t^{m_n} e^{\lambda_n t}: N\in \mN^*,\; c_n, \lambda_n \in \mC, \; m_n \in \mN\bigg\},
 $$
 with pointwise operations of addition, scalar multiplication and multiplication. 
\end{definition}

\noindent In view of the fact  that for ${\rm Re}\; s> {\rm Re}\, \lambda$, the integral
$$F(s):=\int_0^\infty t^n e^{\lambda t}\;e^{-st}\;dt= \frac{n!}{(s-\lambda)^{n+1}}$$
converges absolutely, we see  that 
the Bohl algebra is essentially the class of functions that 
have a strictly proper rational Laplace transform (that is, rational functions for which 
the degree of the denominator is strictly bigger than the degree of the numerator). Thus they  
are linear combinations of sines, cosines, polynomials and exponentials. They arise naturally 
in the study of linear differential equations with constant coefficients since 
$$
\Big( \frac{d}{dt}-\lambda \Big)^{k+1} t^{k} e^{\lambda t}=0.
$$
See for example \cite[Definition~2.5 and Theorem~2.6]{HauStoTre}.

In this article we investigate some algebraic analytic properties of the Bohl algebra. 
Let us also mention that each function in $\textrm{B}$ is the trace on $\mR$ of an entire function
of finite exponential order.  In this light, we therefore obtain also some information on an interesting
subring of  the ring $H(\mC)$ of entire functions.

\begin{definition}[Bass stable rank]
   Let $R$ be  a commutative unital ring  with identity element $1$. We assume that
   $1\not=0$, that is $R$ is not the trivial ring $\{0\}$.

\begin{enumerate}
\item An $n$-tuple $(f_1,\dots,f_n)\in R^n$ is said to be {\it invertible} (or {\it unimodular}), 
  if there exists
 $(x_1,\dots,x_n)\in R^n$ such that the B\'ezout equation 
 $$
 \sum_{j=1}^n x_jf_j=1
 $$
 is satisfied.
   The set of all invertible $n$-tuples is denoted by $U_n(R)$. Note that $U_1(R)=R^{-1}$.
   
  \item An $(n+1)$-tuple $(f_1,\dots,f_n,g)\in U_{n+1}(R)$ is  called {\it reducible}  
if there exists 
 $(a_1,\dots,a_n)\in R^n$ such that 
 $$
 (f_1+a_1g,\dots, f_n+a_ng)\in U_n(R).
 $$

 \item  The {\it Bass stable rank} of $R$, denoted by $\bsr R$,  is the smallest integer $n$ such that every element in $U_{n+1}(R)$ is reducible. 
 If no such $n$ exists, then $\bsr R=\infty$. 
  \end{enumerate}
  \end{definition}
  
 \noindent  Note that if  $\bsr R=n$, where $n<\infty$, and if $m\geq n$,
 then every invertible $(m+1)$-tuple $(f_1,\dots,f_n,g)\in R^{m+1}$ is reducible \cite[Theorem 1]{va}.
Our first main result is the following. 

 \begin{theorem}
 \label{thm_bsr_B_is_infinite}
 The Bass stable rank of the Bohl algebra is infinite.
 \end{theorem}

 \noindent An analogue of the Bass stable rank for topological rings was introduced by Rieffel in \cite{Rie}.
 
 \begin{definition}[Topological stable rank]
Let $R$ be a commutative unital ring endowed with  a topology $\mathcal T$. (We do not assume that the topology
is compatible with the algebraic operations $+$ and $\cdot$ in $R$.) 
 The {\it topological stable rank}, ${\rm tsr}_{\mathcal T}R$, of $(R,\mathcal T)$ is the least integer
  $n$ for which $U_n(R)$ is dense in $R^n$, or infinite if no such $n$ exists.  
\end{definition}

\noindent For the algebra $\textrm{B}$ of Bohl functions, we work with the topology of uniform convergence, that is, 
a basis of open sets is given by the family $(V_{f, \epsilon})_{f\in \textrm{B}, \epsilon>0}$, where 
$$
V_{f, \epsilon}:=\Big\{g \in \textrm{B}: \|f-g\|_\infty:= \displaystyle \sup_{t\in \mR} |f(t)-g(t)|<\epsilon\Big\}.
$$
Our second main result is the following.

\begin{theorem}
\label{thm_tsr_B_is_infinite}
 The topological stable rank of $\textrm{\em B}$ (with respect to the uniform topology) is infinite. 
\end{theorem}
 
 \noindent The organization of the paper is as follows: in Section~\ref{section_bsr} we show that the Bass stable rank of $\textrm{B}$ is infinite, 
 and finally in Section~\ref{section_tsr}, we show that the topological stable rank of $\textrm{B}$ is infinite. 
 
\section{Bass stable rank of $\textrm{B}$ is infinite}
\label{section_bsr}

\begin{definition}[The algebra AP] 
  Let $C_b(\mR,\mC)$ denote the set of bounded, continuous functions on $\mR$ 
  with values in $\mC$, and let AP be the uniform closure 
  in $C_b(\mR,\mC)$ of the set of all functions of the form 
  $$
  Q(t):=\sum_{j=1}^N a_j e^{i \lambda_j t},
  $$
  where $a_j\in \mC$, $\lambda_j\in \mR$ and $N\in \mN^*$. We call $Q$ a {\em generalized trigonometric polynomial}. 
  AP is the set of {\em almost periodic functions}. 
 \end{definition}

 \begin{lemma}
 \label{lemma_ring_homo}
  Let the map $\Psi: \textrm{\em B}\rightarrow {\textrm{\em AP}}$ be defined as follows. If $f\in \textrm{\em B}$, and 
 \begin{equation}
  \label{eq_lemma_ring_homo_1}
 f(t)= \sum_{j=1}^J c_j t^{k_j} e^{(\alpha_j + i \beta_j)t},\quad t\in \mR, 
 \end{equation}
 where the $c_j \in \mC$, $\alpha_j, \beta_j \in \mR$ and $k_j$ are nonnegative integers, 
 then we define $\Psi(f)\in ${\em AP} by 
 $$
 (\Psi(f))(t)=  \sum_{j=1}^J c_j  e^{ i \beta_j t},\quad t\in \mR.
 $$
 Then $\Psi: \textrm{\em B}\rightarrow \textrm{\em AP}$ is well-defined and a ring homomorphism.
 \end{lemma}
\begin{proof} By the linear independence of $\{t^k e^{\lambda t}: \lambda \in \mC, \; k\in \mN\}$, 
it follows that $f$ has a unique decomposition of the form given in \eqref{eq_lemma_ring_homo_1}, and 
so the map $\Psi$ is well-defined. 

 Let $f\in  \textrm{B}$, and $K \subset \mN$, $\Lambda \subset \mC$ be finite sets such that 
 $$
 f(t)= \sum_{\substack{k\in K\\ \lambda \in \Lambda}} f_{k, \lambda} t^{k} e^{\lambda t},
 $$
 where the coefficients $f_{k,\lambda}$ are complex numbers. Let $c, \lambda_* \in \mC$, $k_*\in \mN$. We have two cases:
 
 \smallskip 
 
 \noindent $1^\circ$  $k_* \not\in K$ or $\lambda_* \not\in \Lambda$.  Then 
\begin{eqnarray*}
\Psi(f\!+c t^{k_*} e^{\lambda_* t})&\!\!\!\!=&\!\!\!\! 
\Psi\bigg(\sum_{\substack{k\in K\\ \lambda \in \Lambda}} f_{k, \lambda} t^{k} e^{\lambda t}\!+ c t^{k_*} e^{\lambda_* t}\bigg)
\!=\! 
\sum_{\substack{k\in K\\ \lambda \in \Lambda}} f_{k, \lambda}  e^{i \cdot \textrm{Im}(\lambda)\cdot  t}\!+ 
c  e^{i \cdot \textrm{Im}(\lambda_*)\cdot  t}
\\
&\!\!\!\!=&\!\!\!\! \Psi \bigg( \sum_{\substack{k\in K\\ \lambda \in \Lambda}} f_{k, \lambda} t^{k} e^{\lambda t}\bigg) 
+ \Psi (c t^{k_*} e^{\lambda_* t})
 = \Psi(f)+ \Psi (c t^{k_*} e^{\lambda_* t}).\phantom{\sum_{\substack{k\in K\\ \lambda \in \Lambda}}}
 \end{eqnarray*}

 \noindent $2^\circ$ $k_*\in K$ and $\lambda_* \in \Lambda$. Then
\begin{eqnarray*}
\Psi(f&&\!\!\!\!\!\!\!\!\!\!\!\!\!\!\!   +\;c t^{k_*} e^{\lambda_* t})
\\
&=& \!\!\! \Psi\bigg(\sum_{\substack{k\in K\\ \lambda \in \Lambda}} f_{k, \lambda} t^{k} e^{\lambda t}+ c t^{k_*} e^{\lambda_* t}\bigg)
\\
&=& \!\!\! 
\Psi\bigg(\sum_{\substack{k\in K\setminus \{k_*\}\\ \lambda \in \Lambda\setminus \{\lambda_*\}}} 
\!\!\! f_{k, \lambda} t^{k} e^{\lambda t}+\!\!\! 
\sum_{k\in K\setminus \{k_*\}} \!\!\! 
f_{k, \lambda_*} t^{k} e^{\lambda_* t}+\!\!\! 
\sum_{\lambda \in \Lambda\setminus \{\lambda_*\}} \!\!\! 
f_{k_*, \lambda} t^{k_*} e^{\lambda t}\\
&&\quad \quad +
(f_{k_*, \lambda_*}+c) t^{k_*} e^{\lambda_* t}\bigg)
\\
&=& \!\!\! 
\sum_{\substack{k\in K\setminus \{k_*\}\\ \lambda \in \Lambda\setminus \{\lambda_*\}}} 
 \!\!\! f_{k, \lambda}  e^{i \cdot \textrm{Im}(\lambda)\cdot  t}+ \!\!\! 
\sum_{k\in K\setminus \{k_*\}}  \!\!\! 
f_{k, \lambda_*} e^{i \cdot \textrm{Im}(\lambda_*)\cdot  t} + \!\!\! 
\sum_{\lambda \in \Lambda\setminus \{\lambda_*\}}  \!\!\! 
f_{k_*, \lambda}  e^{i \cdot \textrm{Im}(\lambda)\cdot  t}\\
&&\quad \quad +
(f_{k_*, \lambda_*}+c) e^{i \cdot \textrm{Im}(\lambda_*)\cdot  t}
\\
&=& \sum_{\substack{k\in K\\ \lambda \in \Lambda}} f_{k, \lambda}  e^{i \cdot \textrm{Im}(\lambda)\cdot  t}
+ c  e^{i \cdot \textrm{Im}(\lambda_*)\cdot  t}
 \end{eqnarray*}
\begin{eqnarray*}&=& \Psi \bigg( \sum_{\substack{k\in K\\ \lambda \in \Lambda}} f_{k, \lambda} t^{k} e^{\lambda t}\bigg) + 
\Psi (c t^{k_*} e^{\lambda_* t})
= \Psi(f)+ \Psi (c t^{k_*} e^{\lambda_* t}).
 \end{eqnarray*}
 By the cases $1^\circ$ and $2^\circ$, we have for all  $f,g \in \textrm{B}$, 
 $
 \Psi(f+g)=\Psi(f)+\Psi(g).
 $
 
 \smallskip 
 
 \noindent Now let $g\in  \textrm{B}$, and $L \subset \mN$, $\Omega \subset \mC$ be finite sets such that 
 $$
  g(t)= \sum_{\substack{\ell\in L\\ \omega\in \Omega}} g_{\ell, \omega} t^{\ell} e^{\omega t},
 $$
 where $ g_{\ell, \omega}$ are complex numbers. Then we have 
 \begin{eqnarray*}
  \Psi(f\cdot g)
   \!\!\!\! &=& \!\!\!\!  \Psi\bigg( \sum_{\substack{k\in K\\ \lambda \in \Lambda}}\! f_{k, \lambda} t^{k} e^{\lambda t}
  \cdot \!\sum_{\substack{\ell\in L\\ \omega \in \Omega}} \!g_{\ell, \omega} t^{\ell} e^{\omega t}\!\bigg)\!=\!
  \Psi\bigg( \sum_{\substack{k\in K, \ell \in L\\ \lambda \in \Lambda, \omega \in \Omega}} 
  \!\!\!f_{k, \lambda}g_{\ell, \omega}  t^{k+\ell} e^{(\lambda+\omega) t}\!\bigg)\\
&=&   \sum_{\substack{k\in K, \ell \in L\\ \lambda \in \Lambda, \omega \in \Omega}} 
\Psi\bigg(f_{k, \lambda}g_{\ell, \omega}  t^{k+\ell} e^{(\lambda+\omega) t} \bigg) 
\quad \textrm{ (by additivity of } \Psi \textrm{)}\\
&=& \sum_{\substack{k\in K, \ell \in L\\ \lambda \in \Lambda, \omega \in \Omega}} 
f_{k, \lambda}g_{\ell, \omega}e^{i\cdot \textrm{Im}(\lambda+\omega)\cdot  t}
\\
&=&\sum_{\substack{k\in K\\ \lambda \in \Lambda}} f_{k, \lambda}  e^{i\cdot \textrm{Im}(\lambda)\cdot  t}
  \cdot \sum_{\substack{\ell\in L\\ \omega \in \Omega}} g_{\ell, \omega} e^{i\cdot \textrm{Im}(\omega)\cdot  t}
  = \Psi(f) \cdot \Psi (g) \phantom{\sum_{\substack{k\in K\\ \lambda \in \Lambda}}}
\end{eqnarray*}
\end{proof}

 \noindent We will prove Theorem~\ref{thm_bsr_B_is_infinite} using the following result 
(see  \cite[Theorem~3.5]{MorRup} for a proof).
 
 \begin{proposition}
 \label{prop_MorRup_BSR}
  For all $N\in \mN^*$, there exists an $F=(f_1,\cdots, f_N,g)$, where 
   $f_1, \cdots, f_N, g$ are generalized trigonometric polynomials, 
   which is invertible in $\textrm{\em AP}^n$ with inverse an $(N+1)$-tuple 
    of generalized trigonometric polynomials,
   but not reducible in $\textrm{\em AP}$.
 \end{proposition}

 \noindent We recall here the explicit tuple given in the proof 
 of \cite[Theorem~3.5]{MorRup} below. 
 Given $N\in \mN^*$, let $\{\lambda_1,\cdots, \lambda_{4N}\}$ be any set of positive real numbers that is 
 linearly independent over $\mQ$.  For $j=1,\cdots, 2N$ and $s\in \mN^*$,  let 
 $$
 f_j(t)=(e^{i\lambda_{2j-1} t})^s +(e^{i\lambda_{2j} t})^s -1,
 \quad \textrm{and} \quad  
g(t)=\frac{1}{4}-\sum_{j=1}^N f_j(t) f_{N+j}(t).
 $$
 Then it is clear that  $F:=(f_1,\cdots, f_N, g)$ 
 has an inverse which is an $(N+1)$-tuple of generalized trigonometric polynomials.
 It is shown in  \cite[Theorem~3.5]{MorRup} that 
 $F$ is not reducible in AP.

 \begin{proof}[Proof of Theorem~\ref{thm_bsr_B_is_infinite}] 
 We will prove the claim by contradiction. 
 Suppose that the Bass stable rank of $\textrm{B}$ is at most $N$ for some $N\in \mN$. 
 By Proposition~\ref{prop_MorRup_BSR}, there exists an 
  $F=(f_1,\cdots, f_N, g)\in \textrm{AP}^n $ which is invertible as an 
  element of $\textrm{AP}$ with inverse an $(N+1)$-tuple 
    of generalized trigonometric polynomials, but not reducible in $\textrm{AP}$. 
  Since every generalized trigonometric polynomial is a Bohl function, it follows that 
  $F\in \textrm{B}^{N+1}$ and $F$ is invertible with an inverse $G\in \textrm{B}^{N+1}$. As the Bass stable rank of $\textrm{B}$ is assumed 
  to be $N$, there exist elements $h_1, \cdots , h_N \in \textrm{B}$ and $x_1, \cdots, x_N \in \textrm{B}$  such that 
  $$
  (f_1+h_1 g) x_1+ \cdots + (f_N+h_N g) x_N =1.
  $$
  Applying the ring homomorphism $\Psi$ from Lemma~\ref{lemma_ring_homo}, we obtain 
  $$
  (f_1+ \Psi(h_1) \cdot g) \cdot \Psi(x_1)+\cdots +(f_N+ \Psi(h_N) \cdot g) \cdot \Psi(x_N)=1.
  $$
  But this means that $(f_1, \cdots, f_N, g)$ is reducible in $\textrm{AP}$, a contradiction. 
 \end{proof}

 \begin{remark} 
 The {\it Krull dimension of $R$} is defined to be the
supremum of the lengths of all increasing chains 
 $\mathfrak{p}_0\subset \cdots \subset \mathfrak{p}_n$
of  prime ideals in $R$. An inequality due to Heitmann~\cite{Hei} 
says that the Bass stable rank of a ring is at most the Krull dimension+2, 
and so it follows that the Krull dimension of $\textrm{B}$ is infinite. In particular, 
it is not Noetherian. 
 \end{remark}

 \section{Topological stable rank of $\textrm{B}$ is infinite}
 \label{section_tsr}

\begin{lemma}
\label{lemma_1}
Let $f_{AP}(t):=c_1 e^{i\xi_1 t}+\cdots +c_n e^{i \xi_n t}$, $t\in \mR$, where 
$\xi_1,\cdots, \xi_n $ are distinct  real numbers, and  $c_1,\cdots, c_n$ are nonzero complex numbers. 
Then there exists an $\epsilon>0$ and a  sequence $(t_k)_{k\in \mN}$  such that 
$$
\lim_{k\rightarrow \infty } t_k=+\infty, 
$$
and $|f(t_k)|\geq \epsilon$ for all $k\in \mN$.
\end{lemma}
\begin{proof} By the linear independence of $e^{i\xi_1 t},\cdots, e^{i\xi_n t}$, it follows 
that there is a $t_*\in \mR$ such that $f_{AP}(t_*)\neq 0$. 
Let $2\epsilon:= |f_{AP}(t_*)|$. Then by using the fact that the set of $\epsilon$-translation numbers 
of almost periodic functions are relatively dense, it follows that there exists a $T>0$ 
such that every interval of length $T$ contains an $\epsilon$-translation number of $f_{AP}$ \cite[Theorem~1.9, 1.10]{Cor}. 
Thus, for every $k\in \mN$, the interval $[kT, (k+1)T]$ contains a number $t_k'$ such that 
$\| f_{AP}(\cdot)- f_{AP}(\cdot+t_k')\|_\infty <\epsilon$. But this implies 
$|f_{AP}(t_*)-f_{AP}(t_*+t_k')|\leq \epsilon$, and so with $t_k=t_k'+t_*$, we have 
$|f_{AP}(t_k)|\geq \epsilon$. 
\end{proof}

\begin{lemma}
\label{lemma_a}
If 
\begin{itemize}
 \item[] $p_1,\cdots, p_n$ are polynomials, and 
 \item[] $\xi_1,\cdots, \xi_n$ distinct real numbers, 
\end{itemize}
 such that  
 $
f_I(t):=p_1(t) e^{i\xi_1 t}+\cdots +p_n(t) e^{i\xi_n t},$ $ t\in [0,\infty),
$ 
 satisfies,  for some $M>0$,  
 $$
 \displaystyle \sup_{t\in [0,\infty)} |f_I(t)| \leq M,
 $$
 then the $p_j$  are all constant.
\end{lemma}

\goodbreak 

\begin{proof} We argue by contradiction.  Let us suppose that the polynomials 
$p_j$  are arranged in 
decreasing order of degree, and that the degree of $p_1$ is strictly positive.
 Let $i_*$ be the index such that 
 $$
 d:=\deg(p_1)=\cdots =\deg(p_{i_*})>\deg(p_{i*+1})\geq \cdots\geq \deg(p_{n}).
 $$
 If the degrees all coincide, we put $i^*=n$. 
 If the leading coefficients of $p_1, \cdots, p_{i_*}$ are $c_1, \cdots, c_{i_*}$, then 
  \begin{equation}
\Big|\underbrace{c_{1} e^{i\xi_1 t}+\cdots +c_{i_*} e^{i\xi_{i*} t} }_{=:f_{AP}}
\;+ \;\displaystyle t^{-d}\;\sum_{k,\xi \atop k<d} c_{k,\xi} t^k e^{i\xi t}\Big|\leq \frac{M}{|t|^d}.
\end{equation}
Now if we take $t$ to be the terms of a sequence $(t_k)_{k\in \mN^*}$ corresponding to $f_{AP}$ as in the previous Lemma~\ref{lemma_1}, 
and by passing to the limit as $k\rightarrow \infty$, 
we get the contradiction to the fact that  
$$
\inf_{k\in \mN^*} |f_{AP}(t_k)| =:\epsilon > 0.
$$
So the degree of each $p_j$ must be zero, that is they are constants. 
\end{proof}

\begin{lemma} 
\label{lemma_b}
Let $F:=F_I+F_R$, where 
\begin{eqnarray*}
&& F_I(t):= p_1(t) e^{i\xi_1 t}+\cdots +p_n(t) e^{i\xi_n t},\\
&& F_R(t):= q_1(t) e^{z_1 t}+\cdots + q_m(t) e^{z_m t},\\
&& \xi_1,\cdots, \xi_n \textrm{ are distinct real numbers} ,\\
&& z_1, \cdots, z_m \textrm{ are distinct complex numbers, each with a nonzero real part},\\
&& p_1, \cdots, p_n, q_1,\cdots, q_m \textrm{ are polynomials}.
\end{eqnarray*}
Suppose, moreover, that $\|F\|_\infty <M$ for some $M>0$, 
 where the supremum is taken over $\mR$. 
 Then 
$q_1,\cdots, q_m=0$, that is $F_R\equiv 0$ and $p_1,\cdots, p_n$ are constants.   
Hence $F$ is in $\textrm{\em AP}$. 
\end{lemma}
\begin{proof}
In view of Lemma~\ref{lemma_a}, it is enough to show that $q_1,\cdots, q_m=0$. 
There is no loss of generality in assuming that one of the $z_j$ has a positive real part. 
(Since otherwise, we may just consider the function $\widetilde{F}$ defined by $\widetilde{F}(t)=F(-t)$, 
and then having got the desired conclusion for $\widetilde{F}$, it is also obtained for $F$.) 

Also, we may assume the real numbers Re$(z_j)$ are arranged in decreasing order. Let
 $i_*$ be the index such that 
$$
\textrm{Re}(z_1)=\cdots=\textrm{Re}(z_{i_*})>\textrm{Re}(z_{i_*+1})\geq \cdots \geq \textrm{Re}(z_m),
$$ 
where  $\textrm{Re}(z_1)>0$.  If  the real parts are all the same, we let $i^*=m$.
Let $$
f(t):=q_1(t) e^{i\textrm{Im}(z_1)\cdot  t }+\cdots +q_{i_*} (t) e^{i\textrm{Im}(z_{i_*})\cdot  t}.
$$
Then we have 
$$
e^{\textrm{Re}(z_1)\cdot t} \Big| f(t)
+\frac{q_{i_*+1}(t) e^{z_{i_*+1} t }+\cdots +q_{m}(t)e^{z_m t}}{e^{\textrm{Re}(z_1)\cdot t}}+\frac{f_I(t)}{e^{\textrm{Re}(z_1)\cdot  t}}\Big|\leq M,
$$
where the central sum does not appear if $i^*=m$. 
Hence 
$$
\Big| f(t)
+\frac{q_{i_*+1}(t) e^{z_{i_*+1} t }+\cdots +q_{m}(t)e^{z_m t}}{e^{\textrm{Re}(z_1)\cdot  t}}
+\frac{f_I(t)}{e^{\textrm{Re}(z_1)\cdot  t}}\Big|\leq \frac{M}{e^{\textrm{Re}(z_1)\cdot  t}}.
$$
Since 
$$
\lim_{t\to\infty}q(t)e^{-st}=0
$$ 
for $s>0$,  it follows that $f$ is bounded on $[0,\infty)$, and so by the previous Lemma~\ref{lemma_a}, $q_1,\cdots, q_{i_*}$ 
are constants. 
So in fact $f$ is of the form described in Lemma~\ref{lemma_1}. Hence there exists an $\epsilon>0$ and a sequence $(t_k)_{k\in \mN^*}$ 
such that $t_k\rightarrow \infty$ and $|f(t_k)|\geq \epsilon$. But the above inequality then gives in the limit  
$k\to \infty$  that 
$\epsilon\leq 0,$ a contradiction. 
\end{proof}

\begin{proposition}
 The invertible elements in $\textrm{\em B}$ are the exponentials $c e^{\lambda t}$, 
 where $c\in \mathbb \mC\setminus \{0\}$ and $\lambda\in \mathbb \mC$.
 \end{proposition}
\begin{proof}
Let 
$
f(t)=\displaystyle \sum_{j=1}^n p_j(t)e^{\lambda_jt}\in \textrm{B}.
$ 
Consider the entire function
$$
F(z)=\sum_{j=1}^n p_j(z)e^{\lambda_jz}.
$$
Then, for all $z$ with large modulus, 
$$|F(z)|\leq C|z|^m \sum_{j=1}^n e^{{\rm Re}\; (\lambda_jz)}\leq C |z|^m \sum_{j=1}^ne^{|\lambda_j|\,|z|}\leq e^{b|z|+c}$$
for some  constants $b,c\in \mR^+:=]0,\infty[$.
Now, if $f$ is invertible in $\textrm{B}$, there is $g\in \textrm{B}$ such that $fg=1$ (on $\mR)$.
Being the traces of holomorphic functions, the equality $F(z)G(z)=1$ holds in $\mC$.
In particular, $F$ has no zeros in $\mC$.
Hence, by  Hadamard's classical  theorem on entire functions of exponential type
 (see \cite[p. 84]{LZ}), $F(z)=e^{\beta z +\alpha}$ for some $\alpha,\beta\in \mC$.
\end{proof}

\begin{remark}
A characterization of the invertible  $m$-tuples $(f_1,\dots, f_m)$ in $\textrm{B}$ seems to be ``out of range".
For example, if $(p_1,\dots, p_m)$ are polynomials in $\mC [z_1,\dots,z_n]$ with
$$
\bigcap_{j=1}^m Z(p_j)=\emptyset,
$$
(with $Z(p_j)$ being the set of zeros of  $p_j$ in $\mC^n$), 
then the Hilbert Nullstellensatz 
 gives the existence of polynomials $q_j$
such that 
$$
\sum_{j=1}^m p_jq_j=1.
$$
Now every such polynomial gives rise to
uncountably many different Bohl functions: for example, put 
$$f_j(t):=p_j\big(t, e^{\lambda^{(j)}_1t},\dots, e^{\lambda^{(j)}_{n-1}t}\big),$$
where $\lambda^{(j)}_k\in \mC$ and
$$ g_j(t):=q_j\big(t, e^{\lambda^{(j)}_1t},\dots, e^{\lambda^{(j)}_{n-1}t}\big).$$
Then 
$$
\sum_{j=1}^m p_j(t)q_j(t)=1,
$$
and so we have a solution to a B\'ezout equation in $\textrm{B}$.
\end{remark}

\begin{proof}[Proof of Theorem~\ref{thm_tsr_B_is_infinite} that $\tsr_{\tau_\infty} \textrm{B}=\infty$]  
Let $n\in \mN^*$ and, for $j=1,\dots, n$, 
$$
 f_j(t)=e^{i\lambda_{2j-1} t} +e^{i\lambda_{2j} t} -1.
 $$
  Clearly each $f_j$ is in the Bohl algebra, and hence $F$ belongs to $\textrm{B}^n$. 
By \cite{MorRup}, the
$n$-tuple $F:=(f_1,\dots,f_n)\in {\rm AP}^n$  
cannot be uniformly approximated by invertible tuples in AP.
We claim that this $F$ can't
  be approximated  by invertible tuples in $\textrm{B}^n$. 
 Suppose the contrary. Then 
 there exists a sequence $(G^{(m)})_{m\in \mN^*}$, with $G^{(m)}=(g_1^{(m)},\cdots, g_n^{(m)})
 \in \textrm{B}^n$ for each $m\in \mN^*$, such that 
 $$
 \|f_j-g_j^{(m)}\|_\infty <\frac{1}{m}, \quad j=1,\cdots, n, \; m\in \mN^*.
 $$
 Then, as each $f_j$ is bounded, so is each $g_j^{(m)}$. 
 So  by Lemma~\ref{lemma_b}, each $g_j^{(m)}$ is in AP. Hence $G^{(m)}\in {\textrm{AP}}^n$ for each $m$. 
 We know this $G^{(m)}$ is invertible as an element of $\textrm{B}^n$, and so there exists a $H=(h_1,\cdots, h_n)\in \textrm{B}^n$ 
 such that 
 \begin{equation}
 \label{eq_thm_tsr_B_is_infinite_1}
  G^{(m)}\cdot H=  g_1^{(m)}h_1+\cdots+ g_n^{(m)} h_n=1.
 \end{equation}
 Applying the ring homomorphism $\Psi$ from Lemma~\ref{lemma_ring_homo} in \eqref{eq_thm_tsr_B_is_infinite_1}, we obtain 
 $$
 g_1^{(m)}\cdot \Psi(h_1)+\cdots+ g_n^{(m)} \cdot \Psi(h_n)=1.
 $$
 Hence $G^{(m)}$ is invertible as an element of ${\textrm{AP}}^n$, a contradiction to the fact that 
 $F$ that cannot be approximated by invertible tuples in AP.
\end{proof}

\noindent Recall that the Bass stable rank of the ring $H(\mC)$ of entire functions is one (\cite{pelru}).
Here is now, to the best of our knowledge, the first explicit subring
of $H(\mC)$ having infinite stable rank.

\begin{corollary}
The ring of all functions 
$$
\displaystyle \sum_{j=1}^n c_jz^{n_j} e^{\lambda_jz}, \quad c_j, \lambda_j\in \mC,\;n_j\in \mN,
$$
has Bass  stable  rank infinity.
\end{corollary}


\begin{thebibliography}{10}
 
 \bibitem{Cor}
 C. Corduneanu.
{\em Almost periodic functions.} 
 Chelsea Publ. C., New York,  1989.
 
 \bibitem{HauStoTre}
 H.L. Trentelman, A.A. Stoorvogel, M.L.J. Hautus.
{\em Control theory for linear systems.} 
Communications and Control Engineering Series. Springer, London, 2001. 
 
 \bibitem{Hei}
 R.C. Heitmann.
Generating ideals in Pr\"ufer domains.
{\em Pacific J. Math.}, 62:117-126, no. 1, 1976. 
 
 \bibitem{LZ}
 P.D. Lax, L. Zalcman.
 {\em Complex proofs of real theorems.} 
 American Mathematical Society, 2012
 
 \bibitem{MorRup}
 R. Mortini and R. Rupp. 
 The Bass and topological stable ranks for algebras  of almost periodic functions  on the real line. 
 {\em Transactions of the AMS}, to appear. 
 
\bibitem{pelru}
M. Pelling, L. Rubel,
 Linear compositions of two entire functions, solution to problem  6177, 
{\em American Mathematical Monthly}, 85:505-506, 1978.
 
 \bibitem{Rie}
 M.A. Rieffel. 
 Dimension and stable rank in the K-theory of $C^{*}$-algebras.
{\em Proceedings of the London Mathematical Society}, (3), 46:301-333, no. 2, 1983. 
 
 \bibitem{va} L. Vasershtein.
Stable rank of rings and dimensionality of topological spaces. 
{\em Funct. Anal. Appl.},  5:102-110, 1971;
translation from {\em Funkts. Anal. Prilozh.}, 5:17-27, no. 2, 1971.
 
\end{thebibliography}
 \end{document}